%% file: Hook_Multiplication_in_the_Quantum_K-Theory_of_Grassmannians.tex
\documentclass[10pt]{amsart}
\usepackage{amsmath,amsfonts,amssymb}
\usepackage[colorlinks=true, linkcolor=blue,
  citecolor=blue, urlcolor=blue]{hyperref}
\usepackage{tableau}
\usepackage{subcaption}
\usepackage{mathtools}
\usepackage{float}
\usepackage[alphabetic]{amsrefs}
\usepackage{tikz}
\usepackage{enumitem}

\input{preamble}

\begin{document}

\title
[Hook Multiplication in the Quantum K-Theory of Grassmannians]
{Hook Multiplication in the Quantum K-Theory of Grassmannians}

\author{Joy Hamlin}
\address{Department of Mathematics, Rutgers University, 110
    Frelinghuysen Road, Piscataway, NJ 08854, USA}
\email{joy.hamlin@rutgers.edu}

\begin{abstract}
    We study the quantum K-theory ring $\QK(X)$ of a Grassmannian $X$ and prove a manifestly positive formula for the product of an arbitrary class by a hook class.  This generalizes the quantum K-theoretic Pieri rule, a prior result of Buch and Mihalcea.  We also present a combinatorial interpretation of this result.
\end{abstract}

\maketitle

\section{Introduction}

Let $X = \Gr(m,n)$ be the Grassmannian of $m$ dimensional subspaces of $\C^n$.  Let $\K(X)$ be the K-theory ring of $X$, i.e. the ring generated by classes of vector bundles over $X$ modulo short exact sequences.  The quantum K-theory ring $\widehat{\QK}(X)$ was introduced by Givental
and Lee \cite{givental_wdvv,givental_lee_qktheory,lee_qktheory}.  It is a deformation of $\K(X)$, isomorphic to $\K(X)\otimes_\Z\Z[[q]]$ as a $\Z[[q]]$-module, but with a different multiplicative structure.  For $X^\la$ a Schubert variety in $X$, we denote by $\Oh^\la$ the K-theory class of the structure sheaf of $X^\la$, and the set of all $\Oh^\la$ is a basis for $K(X)$.  The quantum K-theoretic structure constants $N_{\la,\mu}^{\nu,d}$ of $\widehat{\QK}(X)$, which satisfy $\Oh^\la\Oh^\mu=\sum\limits_{\nu,d\geq0}N_{\la,\mu}^{\nu,d}q^d\Oh^\nu$, where $\la,\mu,\nu$ are Young diagrams, are defined from the three point, genus zero, K-theoretic Gromov-Witten invariants of $X$ \cite{buch_qkgrass}.

These structure constants are the object of interest of this paper.  Each $N_{\la,\mu}^{\nu,d}$ is an integer, and their signs satisfy $(-1)^{|\nu|-|\la|-|\mu|-dn}N_{\la,\mu}^{\nu,d}\geq0$ \cite{buch_qkpos}.  However, no manifestly positive combinatorial formula for them is known.  What is known are a Littlewood-Richardson rule for multiplication in K-theory \cite{buch_gamma} and a Pieri rule for multiplication by row or column classes in quantum K-theory \cite{buch_qkgrass}.

This paper proves theorems that generalize both the row and column versions of the quantum K-theory Pieri rules.  To state these theorems, we must set up some notations and definitions.  Let $\hkab$ represent the Young diagram in the shape of a hook with $a+1$ boxes down the first column and $b+1$ boxes across the first row, so that $\hkab$ contains $a+b+1$ boxes total.  Also let $\rho_t$ be the staircase shape $(t-1,t-2,...,1,0)$ in a $t$ by $t$ rectangle, with examples of both $\hkab$ and $\rho_t$ shown in Figure \ref{fig:hkabandrhot}.  Furthermore, write $C_{m,n}(\la,a,b)=N_{\la,\hkab}^{\la,1}$, where this notation emphasizes that $X=\Gr(m,n)$.  If $a<0$ or $b<0$, $\hkab$ is not a valid Young diagram, so we define $C_{m,n}(\la,a,b)$ to be 0 in this case.  Finally, define the number of \textit{quantum corners} of $\la$ to be the maximum of the number of corners of $\la$ and the number of corners of its dual $\la^\vee$, where a corner of a Young diagram is a box with no boxes to its south or east, as shown in Figure \ref{fig:quantumcorners}.  Now we can state:

\begin{figure}[]
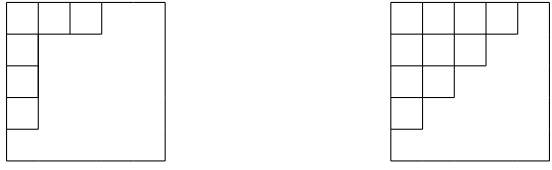

	\[ \tableau{12}
	{	[tlbr]& [tr]& [tr]& [t]& [tr]\\
		[lbr]& [lt]& [t]& []& [r]\\
		[lbr]& [l]& []& []& [r]\\
		[lbr]& []& []& []& [r]\\
		[lb]& [b]& [b]& [b]& [br]
	}
	\hspace{3cm}	
	\tableau{12}
	{	[tlbr]& [tbr]& [tbr]& [tbr]& [tr]\\
		[lbr]& [br]& [br]& []& [r]\\
		[lbr]& [br]& []& []& [r]\\
		[lbr]& []& []& []& [r]\\
		[lb]& [b]& [b]& [b]& [br]
	}\]
	
	\caption{On the left, we see the hook shape $(3\backslash2)$, and on the right, we see the staircase shape $\rho_5$.}
	\label{fig:hkabandrhot}	
\end{figure}

\begin{figure}[]
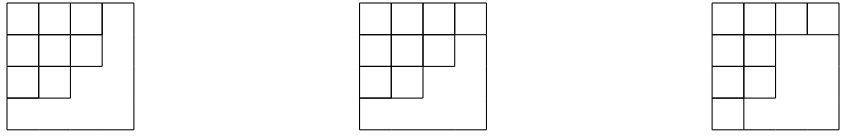

	\[ \tableau{12}
	{	[tlbr]& [tlbr]& [tlbr]& [tlr]\\
		[tlbr]& [tlbr]& [tlbr]& [r]\\
		[tlbr]& [tlbr]& []& [r]\\
		[tlb]& [b]& [b]& [br]
	}
	\hspace{3cm}	
	\tableau{12}
	{	[tlbr]& [tlbr]& [tlbr]& [tlbr]\\
		[tlbr]& [tlbr]& [tlbr]& [r]\\
		[tlbr]& [tlbr]& []& [r]\\
		[tlb]& [b]& [b]& [br]
	}
	\hspace{3cm}	
	\tableau{12}
	{	[tlbr]& [tlbr]& [tlbr]& [tlbr]\\
		[tlbr]& [tlbr]& []& [r]\\
		[tlbr]& [tlbr]& []& [r]\\
		[tlbr]& [b]& [b]& [br]
	}\]
	
	\caption{Each of these diagrams has 3 quantum corners.  On the left, $\la_1=(3,3,2,0)$ has 2 corners and $\la_1^\vee=(4,2,1,1)$ has 3 corners.  In the middle, $\la_2=(4,3,2,0)$ has 3 corners, as does $\la_2^\vee=(4,2,1,0)$.  On the right, $\la_3=(4,2,2,1)$ has 3 corners and $\la_3^\vee=(3,2,2,0)$ has 2 corners.}
	\label{fig:quantumcorners}	
\end{figure}

\begin{thm}\label{thm:main0}
	If $N_{\la,\hkab}^{\nu,d}\neq0$, then either $d=0$, or both $d=1$ and $\nu\subseteq\la$.  Furthermore, if $N_{\la,\hkab}^{\nu,d}\neq0$ and $(\nu,d)\neq(\la,1)$, then $N_{\la,\hkab}^{\nu,d}$ is equal to an explicitly determined structure constant of $\K(X)$.
\end{thm}

The $q\Oh^\la$ term is the maximal term in the product $\Oh^\la\Oh^\hkab$, and is the only term whose coefficient is not given by Theorem \ref{thm:main0}.  Theorems \ref{thm:main1} and \ref{thm:main2} compute this coefficient.

\begin{thm}\label{thm:main1}
	Let $\la$ be a Young diagram with $t$ quantum corners, and let $0\leq a\leq m$, $0\leq b\leq n-m$.  Then there is a reduction:
	
	\[
		C_{m,n}(\lambda,a,b) = C_{t,2t}(\rho_t,a-m+t,b-n+m+t)
	\]
\end{thm}

\begin{thm}
	\label{thm:main2}
	Let $t\geq1$, and let $0\leq a,b\leq t$.  Then:
	
	\[
		C_{t,2t}(\rho_t,a,b)=(-1)^{a+b+1}\sum\limits_{i=1}^{\min(a,b)}\binom{t-1-i}{a-i}\binom{t-1-i}{b-i}\\
	\]
\end{thm}

The remainder of this paper is split into five sections.  The next section describes quantum K-theory in more detail, with emphasis on how to interpret it combinatorially.  The following three sections cover the proofs of the three main theorems, while the last discusses the relationship between classical and quantum structure constants and proposes a combinatorial formula which matches Theorem \ref{thm:main2}.  Theorem \ref{thm:main0} is proved by doing combinatorial analysis to the quantum poset of $X$, Theorem \ref{thm:main1} is proved by reducing $\Gr(m,n)$ to $\Gr(t,2t)$ without changing the value of the relevant structure constant, and Theorem \ref{thm:main2} is proved by inducting on the length of the hook and the number of corners of $\la$.

\section{Quantum K-Theory}

\subsection{The Quantum Poset}

To work combinatorially in $\widehat{\QK}(X)$, we must first define the \textit{quantum poset} of $X$.  Consider $\Z^2$ as a grid of boxes covering the plane, where $(r,c)$ represents the box in row $r$ and column $c$.  By convention row numbers increase going down and column numbers increase going right, to match matrices and Young diagrams.  Declare two boxes $(r,c)$ and $(s,d)$ equivalent if $(r-s,c-d)=k(m,m-n)$ for some $k\in\Z$.  The quantum poset of $X$ is then $\Z^2/\Z(m,m-n)$ ordered by: $[(r,c)]\leq[(s,d)]$ if there exist representatives $(r',c')$ for $[(r,c)]$ and $(s',d')$ for $[(s,d)]$ such that $r'\leq s'$ and $c'\leq d'$.  Intuitively, equivalence classes increase in this order as you travel southeast.

We often identify an element of $\Z^2/\Z(m,m-n)$ with its unique representative whose first coordinate is in $\{1,...,m\}$, and so may refer to an equivalence class as a box.  A \textit{quantum shape} $\la\subset\Z^2/\Z(m,m-n)$ is a (downward-closed, nonempty, proper) order ideal under this partial order, which can also be thought of as an order ideal in $\Z^2$ that respects the equivalence relation.  To specify a quantum shape $\la$, it suffices to specify its maximal elements in rows 1 through $m$.  Specifically, we write $\la=(\la_1,\la_2,...,\la_m)$, where $\la_i$ is the column number of the last box of $\la$ in row $i$. Note that the $\la_i$ are subject to the constraints $\la_i\geq\la_{i+1}$ for each $i$ and $\la_m+n-m\geq\la_1$, or else $\la$ would not be downward-closed.  If each $\la_i$ is between 0 and $n-m$, we identify that quantum shape with a Young diagram and call it a \textit{classical shape}, as shown in Figure \ref{fig:classicalshape}.  Equivalently, $\la$ is a classical shape if the box $[(0,n-m)]=[(m,0)]$ is contained in $\la$ and the box $[(1,n-m+1)]=[(m+1,1)]$ is not.  The special case $\la=(1,1,...,1,0,0,...,0)$ with $a$ 1s is written $(1^a)$, and the special case $\la=(b,0,0,...,0)$ is written $(b)$.

\begin{figure}[]
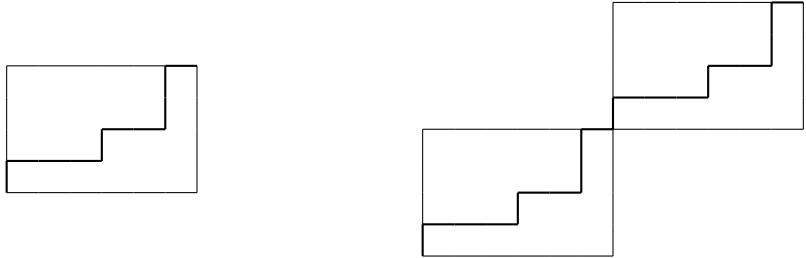

	\[ \tableau{12}
	{	[tl]& [t]& [t]& [t]& [tR]& [Tr]\\
		[l]& []& []& []& [R]& [r]\\
		[l]& []& [R]& [T]& [T]& [r]\\
		[TLb]& [Tb]& [Tb]& [b]& [b]& [br]
	}	
	\hspace{3cm}
	\tableau{12}
	{	[]& []& []& []& []& []& [tl]& [t]& [t]& [t]& [tR]& [Tr]\\
		[]& []& []& []& []& []& [l]& []& []& []& [R]& [r]\\
		[]& []& []& []& []& []& [l]& []& [R]& [T]& [T]& [r]\\
		[]& []& []& []& []& []& [TLb]& [Tb]& [Tb]& [b]& [b]& [br]\\
		[tl]& [t]& [t]& [t]& [tR]& [Tr]& []& []& []& []& []& []\\
		[l]& []& []& []& [R]& [r]& []& []& []& []& []& []\\
		[l]& []& [R]& [T]& [T]& [r]& []& []& []& []& []& []\\
		[TLb]& [Tb]& [Tb]& [b]& [b]& [br]& []& []& []& []& []& []
	}	 \]

	\caption{These are two diagrams of $\la=(5,5,3,0)$ in $\Z^2/\Z(4,-6)$.  In both, the heavy line is the southeastern boundary of $\la$.  On the left we only represent boxes with coordinates from $(1,1)$ to $(4,6)$, and on the right we show a portion of the corresponding order ideal in $\Z^2$.}
	\label{fig:classicalshape}
\end{figure} 

Generalizing the notion of the number of boxes in a Young diagram, we define $|\la|=\sum\limits_{i=1}^m\la_i$ for any shape $\la$.  For any two quantum shapes $\la\subset\nu$, $\nu/\la$ is the \textit{skew shape} consisting of the boxes contained in $\nu$ and not $\la$.  A skew shape is a horizontal (respectively, vertical) strip if it contains no more than one box in each column (respectively, row).  We write $\Row(\nu/\la)\in\{0,...,m\}$ to denote the number of nonempty rows with row numbers 1 through $m$ (where $\nu/\la$ is considered a subset of $\Z^2$), and $\Col(\nu/\la)\in\{0,...,n-m\}$ for the number of nonempty columns with column numbers 1 through $n-m$.

When $\la$ is a classical shape, let $\Oh^\la=[\Oh_{X^\la}]$ be the K-theory class of the structure sheaf of the Schubert variety $X^\la$.  The K-theory ring $\K(X)$ has $\Z$-linear basis given by $\{\Oh^\la:\la\text{ is a classical shape}\}$ (see, for example, \cite{buch_gamma}).  We write $\QK(X)\subset\widehat{\QK}(X)$ to denote $\K(X)\otimes_\Z\Z[q]$, the polynomial submodule of $\widehat{\QK}(X)$.  In fact this submodule is a subring -- i.e., for fixed $\la,\mu$ there exists a $d$ such that $N_{\la,\mu}^{\nu,d'}=0$ for any $\nu$ and any $d'\geq d$ \cite{buch_qkgrass}.  Then the set $\{q^d\Oh^\la:d\geq0,\la\text{ is a classical shape}\}$ is a $\Z$-linear basis for $\QK(X)$, and $\{q^d\Oh^\la:d\in\Z,\la\text{ is a classical shape}\}$ is a basis for the localization $\QK(X)_q$.

To see another way of expressing this basis of $\QK(X)$, we introduce a shift operation.  For any quantum shape $\la$ and any $d\in\Z$, define $\la[d]=\{(a+d,b+d):(a,b)\in\la\}$ to be the shape resulting from shifting the boundary of $\la$ by $d$ steps down and to the right.  For any northwest-to-southeast diagonal, there must be a maximal box that is contained in both that diagonal and $\la$, since $\la$ is nonempty and proper.  Shifting $\la$ by 1 causes the next box on that diagonal to be maximal in $\la$.  Thus there is always exactly one way to shift $\la$ such that the box $(0,n-m)$ is maximal in $\la$, so for any $\la$ there exists a unique $d$ such that $\la[-d]$ is a classical shape.  We then define $\Oh^\la=q^d\Oh^{\la[-d]}$, for any quantum shape $\la$.  This allows the basis of $\QK(X)_q$ to be equivalently expressed as $\{\Oh^\la:\la\text{ is a quantum shape}\}$, allowing $\QK(X)_q$ to be thought of purely combinatorially, without explicit reference to $q$.

For quantum shapes $\la,\mu,\nu$, we now write $N_{\la,\mu}^\nu$ for the structure constants of $\QK(X)_q$, instead of requiring that they be classical shapes and writing $d$ explicitly.  In other words, $\Oh^\la\Oh^\mu=\sum\limits_\nu N_{\la,\mu}^\nu\Oh^\nu$, where $\nu$ ranges over all quantum shapes.  In the case where all three are classical shapes, $N_{\la,\mu}^\nu=N_{\la,\mu}^{\nu,0}$ is a structure constant of $\K(X)$.

\subsection{The Quantum Pieri Rule}

Multiplication in $\QK(X)$ is fully determined by the quantum K-theoretic Pieri rule.  Here we present both the row and column forms of this rule, though once one is proven the other follows from the involution between $\Gr(m,n)$ and $\Gr(n-m,n)$.  The classical K-theory Pieri rule was originally proved by Lenart (\cite{Lenart_option1}), and the quantum K-theory Pieri rule has the exact same form when described using the quantum poset.  A version of the quantum Pieri rule in terms of the $q^d\Oh^\la$ form of the basis was proved by Buch and Mihalcea in \cite{buch_qkgrass}, and the version we state is from \cite{buch_qkpieri}.

\begin{thm}[\cite{Lenart_option1}, \cite{buch_qkgrass}, \cite{buch_qkpieri}]\label{thm:pieri} Let $\la$ be a quantum shape in $\Z^2/\Z(m,m-n)$, and let $0\leq i\leq m,0\leq j\leq n-m$.  Then:
	\begin{align*}
		\Oh^\la\Oh^j & =\sum\limits_\nu(-1)^{|\nu/\la|-j}\binom{\Row(\nu/\la)-1}{|\nu/\la|-j}\Oh^\nu\\
		\Oh^\la\Oh^{1^i} & =\sum\limits_\mu(-1)^{|\mu/\la|-i}\binom{\Col(\mu/\la)-1}{|\mu/\la|-i}\Oh^\mu
	\end{align*}
	
	Where the first sum is over all quantum shapes $\nu$ such that $\nu/\la$ is a horizontal strip, and the second is over all quantum shapes $\mu$ such that $\mu/\la$ is a vertical strip.  
\end{thm}

\subsection{Hook Multiplication in $\QK(X)$}

Recall that $\hkab$ is the hook shape Young diagram with $a+1$ boxes down the first column and $b+1$ across the first row, for $a+b+1$ boxes total.  We now consider it as a classical shape in the quantum poset, and this paper proves a formula for the structure constants in the quantum product $\Oh^\la\Oh^\hkab$ for any quantum shape $\la$.  Note that $\hk{a}{0}=\Oh^{1^{a+1}}$ and $\hk{0}{b}=\Oh^{b+1}$.  All terms in $\Oh^\la\Oh^\hkab$ except for the coefficient $N_{\la,\hkab}^{\la[1]}$ of the $q\Oh^\la$ term can be directly computed using the ordinary K-theoretic Littlewood-Richardson rule, so most of our efforts focus on this coefficient.  As was defined in the introduction, we use the notation $C_{m,n}(\la,a,b)=N_{\la,\hkab}^{\la[1]}$ for these particular structure constants of $\QK(\Gr(m,n))$, and the notation $\rho_t$ for the staircase Young diagram, now considered as a quantum shape in $\Z^2/(t,-t)$ and shown in Figure \ref{fig:rhoexample}.  We also see the motivation for the definition of a quantum corner here -- the number of quantum corners a Young diagram has is equal to the number of maximal boxes in the corresponding classical shape.

\begin{figure}[]
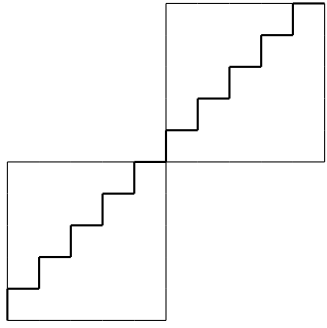

	\[ \tableau{12}
	{	[]& []& []& []& []& [lt]& [t]& [t]& [tR]& [Tr]\\
		[]& []& []& []& []& [l]& []& []& [TL]& [r]\\
		[]& []& []& []& []& [l]& []& [TL]& []& [r]\\
		[]& []& []& []& []& [l]& [TL]& []& []& [r]\\
		[]& []& []& []& []& [TLb]& [b]& [b]& [b]& [br]\\
		[tl]& [t]& [t]& [tR]& [Tr]& []& []& []& []& []\\
		[l]& []& [R]& [T]& [r]& []& []& []& []& []\\
		[l]& [R]& [T]& []& [r]& []& []& []& []& []\\
		[lRB]& [T]& []& []& [r]& []& []& []& []& []\\
		[Lb]& [b] &[b]&[b]&[br]& []& []& []& []& []
	}\]
	\caption{The shape $\rho_5$ as an element of $\Z^2/\Z(5,-5)$.  It has 4 corners when considered as a classical shape, but 5 quantum corners.}
	\label{fig:rhoexample}
\end{figure}

\section{Proof of Theorem \ref{thm:main0}}

K-theoretic Schubert classes given by hook shapes are unique in that they can be decomposed into terms that multiply at most a row by a column.  The following is an explicit description of how.

\begin{lemma}\label{lemma:lincombo}
	Fix $a,b\geq1$.  Then in both $\K(X)$ and $\QK(X)$:
	
	\begin{align}\label{eqn:lincombo}
	\Oh^\hkab=\sum\limits_{i=2}^{a+1}\binom{a-i+b}{b-1}\Oh^{1^i}+\sum\limits_{j=2}^{b+1}\binom{a+b-j}{a-1}\Oh^j-\sum\limits_{i=1}^a\sum\limits_{j=1}^b\binom{a-i+b-j}{a-i}\Oh^{1^i}\Oh^j
	\end{align}
\end{lemma}
\begin{proof}
	We use the row Pieri rule to compute the product $\Oh^{1^i}\Oh^j$:
	
	\begin{align}\label{eqn:pierioutput}
		\Oh^{1^i}\Oh^j &=\hk{i-1}{j}+\hk{i}{j-1}-\hk{i}{j}
	\end{align}

	It is immediately clear from the Pieri rule that the coefficients on the existing terms in equation (\ref{eqn:pierioutput}) are correct.  To see that no other terms can exist in this product, notice that if $\nu/(1^i)$ is a horizontal strip it must have at most two nonempty rows, so $\Row(\nu/(1^i))-1$ is either 1 or 0.  If it is 0 we must have  $|\nu/(1^i)|=j$, so the only possible shape $\nu$ can be is $(i-1\backslash j)$.  If it is 1 we have either $|\nu/(1^i)|=j$ or $|\nu/(1^i)|=j+1$, and the only possible shapes $\nu$ can be are $(i\backslash j-1)$ and $(i\backslash j)$ respectively.
	
	Now we induct to prove the lemma.  In the base case $a=b=1$, equation (\ref{eqn:lincombo}) reduces to $\hk{1}{1}=\Oh^{1^2}+\Oh^2-\Oh^1\Oh^1$, which is exactly equation (\ref{eqn:pierioutput}).  Thus the base case holds.
	
	If $a=1$, equation \ref{eqn:lincombo} reduces to $\hk{1}{b}=\Oh^{1^2}+\sum\limits_{j=2}^{b+1}\Oh^j-\sum\limits_{j=1}^b\Oh^1\Oh^j$.  Inductively assume that, for some $b\geq2$, this holds for $b-1$:
	
	\begin{align*}
		\hk{1}{b} & = \hk{0}{b}+\hk{1}{b-1}-\Oh^{1^1}\Oh^b\\
		& = \Oh^{b+1}+\Oh^{1^2}+\sum\limits_{j=2}^b\Oh^j-\sum\limits_{j=1}^{b-1}\Oh^1\Oh^j-\Oh^1\Oh^b\\
		& = \Oh^{1^2}+\sum\limits_{j=2}^{b+1}\Oh^j-\sum\limits_{j=1}^b\Oh^1\Oh^j\\
	\end{align*}

	Which completes the induction in this case.  The proof for $b=1$ follows from the involution between $\Gr(m,n)$ and $\Gr(n-m,n)$.  Finally, we assume $a,b>1$ and induct over the sum $a+b$, completing the proof of the lemma.
	
	\begin{align*}
		\hk{a}{b} & = \hk{a-1}{b}+\hk{a}{b-1}-\Oh^{1^a}\Oh^b\\
		& = -\Oh^{1^a}\Oh^b+\sum\limits_{i=2}^a\binom{a-i+b-1}{b-1}\Oh^{1^i}+\sum\limits_{j=2}^{b+1}\binom{a+b-j-1}{a-2}\Oh^j\\
		& \hspace{5mm}-\sum\limits_{i=1}^{a-1}\sum\limits_{j=1}^b\binom{a-i+b-j-1}{a-i-1}\Oh^{1^i}\Oh^j+\sum\limits_{i=2}^{a+1}\binom{a-i+b-1}{b-2}\Oh^{1^i}\\
		& \hspace{5mm}+\sum\limits_{j=2}^b\binom{a+b-j-1}{a-1}\Oh^j-\sum\limits_{i=1}^a\sum\limits_{j=1}^{b-1}\binom{a-i+b-j-1}{a-i}\Oh^{1^i}\Oh^j\\
		& =\sum\limits_{i=2}^{a+1}\binom{a-i+b}{b-1}\Oh^{1^i}+\sum\limits_{j=2}^{b+1}\binom{a+b-j}{a-1}\Oh^j-\sum\limits_{i=1}^a\sum\limits_{j=1}^b\binom{a-i+b-j}{a-i}\Oh^{1^i}\Oh^j\\
	\end{align*}
\end{proof}

Define a nonempty skew shape $\nu/\la$ to be a \textit{rim} if $\nu\subseteq\la[1]$.  This is equivalent to saying $\nu$ is the result of adding a horizontal strip and a vertical strip to $\la$, or to saying that $\nu/\la$ contains no 2 by 2 squares.  We say that $\nu/\la$ is an \textit{unbroken} rim if $\nu=\la[1]$, and is a \textit{broken} rim if $\nu\subsetneq\la[1]$.

\begin{cor}\label{cor:atmostq}
	If $\Oh^\nu$ has nonzero coefficient in $\Oh^\la\Oh^\hkab$, then $\nu/\la$ is a rim.
\end{cor}
\begin{proof}	
	Let $\nu$ be such that $\Oh^\nu$ has nonzero coefficient in $\Oh^\la\Oh^\hkab$.  By Lemma \ref{lemma:lincombo}, $\hk{a}{b}$ is a linear combination of products $\Oh^{1^i}\Oh^j$, $i,j\geq0$.  Then $\Oh^\la\Oh^\hkab$ is a linear combination of $\Oh^\la\Oh^{1^i}\Oh^j$.  By the Pieri rules, all terms in the expansion of $\Oh^\la\Oh^{1^i}\Oh^j$ are obtained by adding a horizontal and a vertical strip to $\la$.  $\Oh^\nu$ is one such term, so $\nu$ is the result of adding a horizontal and a vertical strip to $\la$ -- i.e., $\nu/\la$ is a rim.
\end{proof}

We now consider translating quantum shapes in the quantum poset. This is useful because structure constants for classical shapes are known by the K-theoretic Littlewood-Richardson rule, and since $\la$ is classical exactly when $(0,n-m)$ is maximal in $\la$, any shape can be translated to become a classical shape. As we see in Lemma \ref{lemma:sidel}, some translations preserve structure constants, showing that particular
quantum structure constants are equal to classical structure constants.

\begin{lemma}\label{lemma:sidel}
	If $\la'$ is the translation of $\la$ by $r$ rows down and $s$ columns to the right, and $\nu'$ is the same translation of $\nu$, then $N_{\la,\mu}^\nu=N_{\la',\mu}^{\nu'}$.
\end{lemma}
\begin{proof}
	From the Pieri rules, we see that multiplication by the Sidel class $\Oh^{1^m}$ translates $\la$ one column to the right, and multiplication by the Sidel class $\Oh^{n-m}$ translates $\la$ one row down.  Then $\Oh^{\la'}=(\Oh^{1^m})^s(\Oh^{n-m})^r\Oh^\la$, and:
	
	\begin{align*}
		\Oh^{\la'}\Oh^\mu &= (\Oh^{1^m})^s(\Oh^{n-m})^r\Oh^\la\Oh^\mu\\
		&= (\Oh^{1^m})^s(\Oh^{n-m})^r\sum\limits_\nu N_{\la,\mu}^\nu\Oh^\nu\\
		&= \sum\limits_\nu N_{\la,\mu}^\nu\left((\Oh^{1^m})^s(\Oh^{n-m})^r\Oh^\nu\right)\\
		&= \sum\limits_\nu N_{\la,\mu}^\nu\Oh^{\nu'}
	\end{align*}

	This means the coefficient on $\Oh^{\nu'}$ in $\Oh^{\la'}\Oh^\mu$ is precisely $N_{\la,\mu}^\nu$, as desired.
\end{proof}

We can now prove Theorem \ref{thm:main0}.  The statement of Theorem \ref{thm:main0} given in the introduction is for $\la,\nu$ classical shapes, but in fact this proof works when they are arbitrary quantum shapes.  Restating the theorem in the language of quantum shapes and slightly strengthening it, it becomes Theorem \ref{thm:main0quantum}.

\begin{thm}\label{thm:main0quantum}
	If $N_{\la,\hkab}^\nu\neq0$, then $\la\subseteq\nu\subseteq\la[1]$.  Furthermore, if $N_{\la,\hkab}^\nu\neq0$ and $\nu\neq\la[1]$, then $N_{\la,\hkab}^\nu$ is equal to an explicitly determined structure constant of $\K(X)$.
\end{thm}

To see that Theorem \ref{thm:main0quantum} implies Theorem \ref{thm:main0}, assume that $\nu,\la$ are classical shapes and $a,b,d$ are integers with $N_{\la,\hkab}^{\nu,d}\neq0$.  Then $N_{\la,\hkab}^{\nu[d]}\neq0$, since $q^d\Oh^\nu=\Oh^{\nu[d]}$.  

We must have $\la\subseteq\nu[d]$, since multiplication by a hook never removes boxes from a quantum shape.  Since $\la$ and $\nu$ are classical $(0,n-m)$ is maximal in each of them.  Since $(0,n-m)\in\la\subset\nu[d]$ and $(0,n-m)$ is maximal in $\nu$, this means $d\geq0$.  Now by Theorem \ref{thm:main0quantum}, $\nu[d]\subseteq\la[1]$.  This means $\nu\subseteq\la[1-d]$, so by the same reasoning as above $1-d\geq0$.  Thus either $d=0$ or $d=1$.  If $d=1$, we have $\nu[1]\subseteq\la[1]$, so $\nu\subseteq\la$.  This proves the first sentence of Theorem \ref{thm:main0}.  For the second sentence, it suffices to notice that $(\nu,d)=(\la,1)$ iff $\nu=\la[1]$.

\begin{proof}[Proof of Theorem \ref{thm:main0quantum}]
	By Corollary \ref{cor:atmostq}, $\nu\subseteq\la[1]$.  Since multiplication by a hook does not remove boxes from a quantum shape, $\la\subseteq\nu$.  This proves the first sentence of the theorem.
	
	Now assume $N_{\la,\hkab}^\nu\neq0$ and $\nu\neq\la[1]$.  This means $\nu\subsetneq\la[1]$, so $\nu/\la$ is a broken rim.  Choose a maximal box $(r,c)\in\la[1]$ such that $(r,c)\notin\nu$.  By Lemma \ref{lemma:sidel}, we can translate $\la$ and $\nu$ by the same amount, and $N_{\la,\hkab}^\nu$ will not change.  Define $\nu'$ and $\la'$ to be the shapes obtained from $\nu$ and $\la$ respectively under the translation that sends $(r,c)$ to $(1,n-m+1)$.  Since $(r,c)$ is maximal in $\la[1]$, $(1,n-m+1)$ is maximal in $\la'[1]$, which in turn means $(0,n-m)$ is maximal in $\la'$.  Thus $\la'$ is a classical shape.  Since $\la\subset\nu$, $\la'\subset\nu'$, so $(0,n-m)\in\nu'$ as well.  Furthermore, since $(r,c)\notin\nu$, $(1,n-m+1)\notin\nu'$, and so $\nu'$ is a classical shape.  Since $\la'$ and $\nu'$ are both classical shapes, $N_{\la,\hkab}^\nu=N_{\la',\hkab}^{\nu'}$ is an explicit structure constant of the ordinary K-theory  ring, as desired.
\end{proof}

\section{Proof of Theorem \ref{thm:main1}}

Now we can use the Pieri rules to find the $q\Oh^\la$ coefficient of $\Oh^\la\Oh^\hkab$.  Recall that we define $C_{m,n}(\la,a,b)=N^{\la[1]}_{\la,\hkab}$, with $C_{m,n}(\la,a,b)=0$ if $a<0$ or $b<0$.  The formulae presented in this section do not have manifestly determined signs; Theorem \ref{thm:main2} covers turning them into manifestly positive sums.

\begin{lemma}\label{lemma:longformula}
	When $\la$ has $t$ quantum corners and $a,b\geq0$, $$C_{m,n}(\la,a,b)=\sum\limits_{i=1}^a\sum\limits_{j=1}^b(-1)^{n-i-j-1}\binom{a-i+b-j}{a-i}\binom{t-1}{m-i}\binom{t-1}{n-m-j}$$
\end{lemma}
\begin{proof}
	Writing $\Oh^\hkab$ as a linear combination of $\Oh^{1^i}$, $\Oh^j$, and $\Oh^{1^i}\Oh^j$, we see that it suffices to find the coefficient on $q\Oh^\la$ in $\Oh^\la\Oh^{1^i}$, $\Oh^\la\Oh^j$, and $\Oh^\la\Oh^{1^i}\Oh^j$.  The shape of $q\Oh^\la$ is not $\la$ with a vertical or horizontal strip added, so by the row and column Pieri rules the $\Oh^\la\Oh^{1^i}$ and $\Oh^\la\Oh^j$ terms contribute nothing and we focus on $\Oh^\la\Oh^{1^i}\Oh^j$.  If $a=0$ or $b=0$, then $\hkab$ is a horizontal or vertical strip, so there are no $\Oh^\la\Oh^{1^i}\Oh^j$ terms and $C_{m,n}(\la,a,b)=0$.  If $a,b\geq1$, then by Lemma \ref{lemma:lincombo}, it suffices to show that the coefficient of $q\Oh^\la$ in $\Oh^\la\Oh^{1^i}\Oh^j$ is $(-1)^{n-i-j}\binom{t-1}{m-i}\binom{t-1}{n-m-j}$.  
	
	The shape of $q\Oh^\la$ is $\la$ with a vertical strip of $m$ boxes and a horizontal strip of $n-m$ boxes added on.  Therefore, the only way to create this shape from $\la$ using one application of the row Pieri rule and one application of the column Pieri rule is if both are introducing a strip of maximum length.
	
	When applying the column Pieri rule, we need to find how many nonempty columns the the vertical strip has.  Each such column contains exactly one corner box of $\la$ and every corner box of $\la$ gives rise to such a column, since we are adding a vertical strip of maximum length.  Thus there are $t$ nonempty columns, so the multiplication by $\Oh^{1^i}$ contributes a factor of $(-1)^{m-i}\binom{t-1}{m-i}$.  Similarly, by associating a row with the unique corner in that row, the multiplication by $\Oh^j$ contributes a factor of $(-1)^{n-m-j}\binom{t-1}{n-m-j}$.  This gives a coefficient of $(-1)^{n-i-j}\binom{t-1}{m-i}\binom{t-1}{n-m-j}$ on $q\Oh^\la$ in $\Oh^\la\Oh^{1^i}\Oh^j$, as desired.
\end{proof}

\begin{cor}\label{cor:taut}
	$$C_{t,2t}(\rho_t,a,b)=\sum\limits_{i=1}^a\sum\limits_{j=1}^b(-1)^{i+j+1}\binom{a-i+b-j}{a-i}\binom{t-1}{i-1}\binom{t-1}{j-1}$$
\end{cor}
\begin{proof}
	$\rho_t$ has $t$ corners, and $m=t,n=2t$.  This corollary then follows directly from Lemma \ref{lemma:longformula}.
\end{proof}

We are now ready to recall and prove Theorem \ref{thm:main1}.

\begin{theorem2*}
	Let $\la$ be a Young diagram with $t$ quantum corners, and let $0\leq a\leq m$, $0\leq b\leq n-m$.  Then there is a reduction:
	
	\[
	C_{m,n}(\lambda,a,b) = C_{t,2t}(\rho_t,a-m+t,b-n+m+t)
	\]
\end{theorem2*}

\begin{proof}
	Set $\alpha=a-m+t$ and $\beta=b-n+m+t$.  The equation we need to prove becomes $C_{m,n}(\la,\alpha+m-t,\beta+n-m-t)=C_{t,2t}(\rho_t,\al,\be)$, and the bounds for $\al$ and $\be$ are $-m+t\leq\al\leq t$ and $-n+m+t\leq\be\leq t$.  Stating the result of Lemma $\ref{lemma:longformula}$ in terms of $\al$ and $\be$:
	
	\begin{align*}
		C_{m,n}&(\la,\al+m-t,\be+n-m-t)\\
		& = \sum\limits_{i=1}^{\al+m-t}\sum\limits_{j=1}^{\be+n-m-t}(-1)^{n-i-j-1}\binom{\al-i+\be-j+n-2t}{\al+m-t-i}\binom{t-1}{m-i}\binom{t-1}{n-m-j}\\
	\end{align*}
	
	If $1\leq i\leq m-t$, then $t-1<m-i$ and thus $\binom{t-1}{m-i}=0$.  Similarly, if $1\leq j\leq n-m-t$, $\binom{t-1}{n-m-j}=0$.  This means the lower bounds on the sum can be replaced with $i=m-t+1,j=n-m-t+1$.  If $\al\leq0$ or $\be\leq0$, the sum is now empty, so $C_{m,n}(\la,\al+m-t,\be+n-m-t)=0$.  In that case, $C_{t,2t}(\rho_t,\al,\be)=0$ as well, so the result is proven.
	
	If the sum is not empty, after replacing the lower bounds we change variables $k=i-m+t,l=j-n+m+t$.
	
	\begin{align*}
		C&_{m,n}(\la,\al+m-t,\be+n-m-t)\\
		& =  \sum\limits_{i=m-t+1}^{\al+m-t}\sum\limits_{j=n-m-t+1}^{\be+n-m-t}(-1)^{n-i-j-1}\binom{\al-i+\be-j+n-2t}{\al+m-t-i}\binom{t-1}{m-i}\binom{t-1}{n-m-j}\\
		& = \sum\limits_{k=1}^\al\sum\limits_{l=1}^\be(-1)^{-k-l+2t-1}\binom{\al-k+\be-l}{a-i}\binom{t-1}{t-k}\binom{t-1}{t-l}\\
		& = \sum\limits_{k=1}^\al\sum\limits_{l=1}^\be(-1)^{k+l+1}\binom{\al-k+\be-l}{a-k}\binom{t-1}{k-1}\binom{t-1}{l-1}\\
	\end{align*}
	
	The last line above is exactly equal to $C_{t,2t}(\rho_t,\al,\be)$ as found in Corollary \ref{cor:taut}, so we are done.
\end{proof}

Theorem \ref{thm:main1} is equivalent to the following statement: 

\begin{remark}
	\label{rem:reduction}
	Suppose $\la$ has $t$ corners, where $t<\max(m,m-n)$.  Then we can remove a repeated row or column from $\la$ to get a new shape $\la'$.  Set $n'=n-1$.  If we removed a row, set $m'=m-1$, $a'=a-1$, and $b'=b$, and if we removed a column set $m'=m$, $a'=a$, and $b'=b-1$.  Then $C_{m,n}(\la,a,b)=C_{m',n'}(\la',a',b')$.
\end{remark}

\begin{figure}[H]
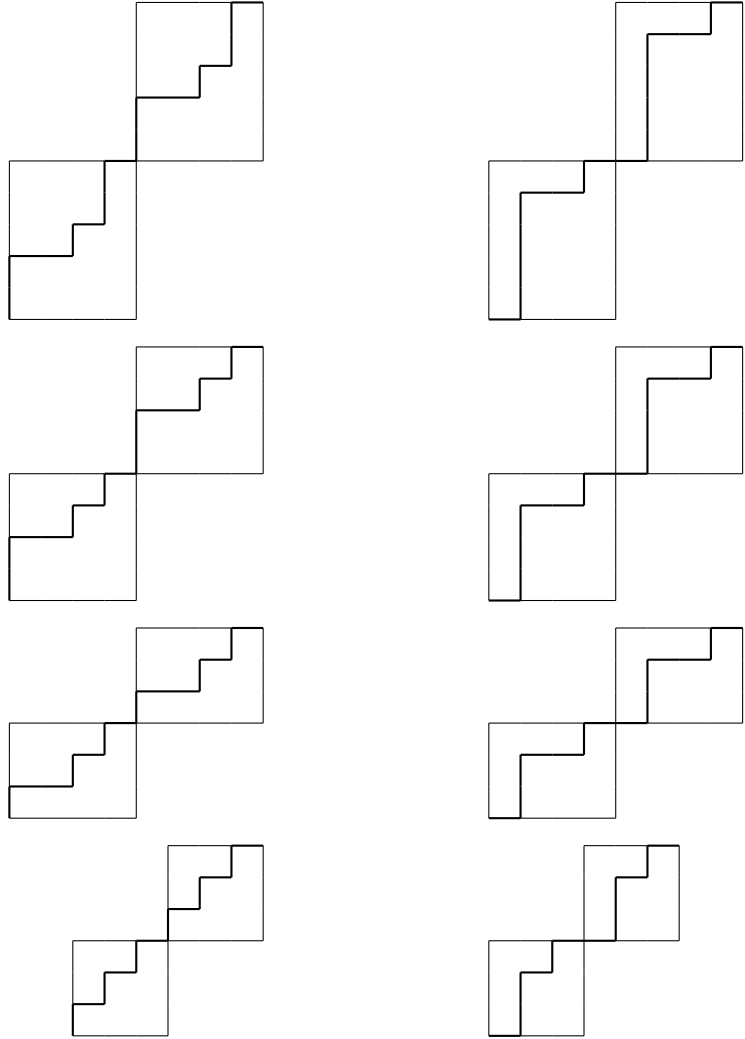

	\[ \tableau{12}
	{	[]& []& []& []& [tl]& [t]& [tR]& [Tr]\\
			[]& []& []& []& [l]& []& [R]& [r]\\
			[]& []& []& []& [l]& [R]& [T]& [r]\\
			[]& []& []& []& [LT]& [T]& []& [r]\\
			[]& []& []& []& [Lb]& [b] &[b]&[br]\\
			[tl]& [t]& [tR]& [Tr]& []& []& []& []\\
			[l]& []& [R]& [r]& []& []& []& []\\
			[l]& [R]& [T]& [r]& []& []& []& []\\
			[LT]& [T]& []& [r]& []& []& []& []\\
			[Lb]& [b] &[b]&[br]& []& []& []& []
		}
	\hspace{3cm}
	\tableau{12}
	{	[]& []& []& []& [tl]& [t]& [tR]& [Tr]\\
		[]& []& []& []& [l]& [LT]& [T]& [r]\\
		[]& []& []& []& [l]& [L]& []& [r]\\
		[]& []& []& []& [lR]& []& []& [r]\\
		[]& []& []& []& [lBR]& [b] &[b]&[br]\\
		[tl]& [t]& [tR]& [Tr]& []& []& []& []\\
		[l]& [LT]& [T]& [r]& []& []& []& []\\
		[l]& [L]& []& [r]& []& []& []& []\\
		[lR]& []& []& [r]& []& []& []& []\\
		[lBR]& [b] &[b]&[br]& []& []& []& []
	} \]

	\[ \tableau{12}
	{	[]& []& []& []& [tl]& [t]& [tR]& [Tr]\\
			[]& []& []& []& [l]& [R]& [T]& [r]\\
			[]& []& []& []& [LT]& [T]& []& [r]\\
			[]& []& []& []& [Lb]& [b] &[b]&[br]\\
			[tl]& [t]& [tR]& [Tr]& []& []& []& []\\
			[l]& [R]& [T]& [r]& []& []& []& []\\
			[LT]& [T]& []& [r]& []& []& []& []\\
			[Lb]& [b] &[b]&[br]& []& []& []& []
		}	
	\hspace{3cm}
	\tableau{12}
	{	[]& []& []& []& [tl]& [t]& [tR]& [Tr]\\
		[]& []& []& []& [l]& [LT]& [T]& [r]\\
		[]& []& []& []& [lR]& []& []& [r]\\
		[]& []& []& []& [lBR]& [b] &[b]&[br]\\
		[tl]& [t]& [tR]& [Tr]& []& []& []& []\\
		[l]& [LT]& [T]& [r]& []& []& []& []\\
		[lR]& []& []& [r]& []& []& []& []\\
		[lBR]& [b] &[b]&[br]& []& []& []& []
	} \]

	\[ \tableau{12}
	{	[]& []& []& []& [tl]& [t]& [tR]& [Tr]\\
		[]& []& []& []& [l]& [R]& [T]& [r]\\
		[]& []& []& []& [LbT]& [Tb] &[b]&[br]\\
		[tl]& [t]& [tR]& [Tr]& []& []& []& []\\
		[l]& [R]& [T]& [r]& []& []& []& []\\
		[LbT]& [Tb] &[b]&[br]& []& []& []& []
	}	
	\hspace{3cm}
	\tableau{12}
	{	[]& []& []& []& [tl]& [t]& [tR]& [Tr]\\
		[]& []& []& []& [l]& [LT]& [T]& [r]\\
		[]& []& []& []& [lBR]& [b] &[b]&[br]\\
		[tl]& [t]& [tR]& [Tr]& []& []& []& []\\
		[l]& [LT]& [T]& [r]& []& []& []& []\\
		[lBR]& [b] &[b]&[br]& []& []& []& []
	} \]
	
	\[ \tableau{12}
	{	[]& []& []& [tl]& [tR]& [Tr]\\
		[]& []& []& [lR]& [T]& [r]\\
		[]& []& []& [LbT]& [b]& [br]\\
		[tl]& [tR]& [Tr]& []& []& []\\
		[lR]& [T]& [r]& []& []& []\\
		[LbT]& [b]& [br] & []& []& []
	}	
	\hspace{3cm}
	\tableau{12}
	{	[]& []& []& [tl]& [tR]& [Tr]\\
		[]& []& []& [l]& [LT]& [r]\\
		[]& []& []& [lBR]& [b]&[br]\\
		[tl]& [tR]& [Tr]& []& []& []\\
		[l]& [LT]& [r]& []& []& []\\
		[lBR]& [b]&[br]& []& []& []
	} \]
	
	\caption{We consider the setup in Remark \ref{rem:reduction} where $m=5$, $n=9$, $\la=(3,3,2)$, $t=3$, $a=4$, $b=2$.  The diagrams on the left show successive stages of removing a repeated row or column from $\la$, and the diagrams on the right correspondingly remove a row or column from $(a\backslash b)$.  This process terminates when $\la$ has become $\rho_3$ and $(a\backslash b)$ has become $(a-m+t,b-n+m+t)$.}
	\label{fig:reduction}
\end{figure}

Intuitively, we have removed a row or column from the $m$ by $n-m$ rectangle that $\la$ lives in, as well as from the hook $\hkab$.  $\la$ and $\la'$ both have the same number of quantum corners $t$, and by construction $a-m+t=a'-m'+t$ and $b-n+m+t=b'-n'+m'+t$.  Thus Theorem \ref{thm:main1} implies Remark \ref{rem:reduction}, because both coefficients are equal to $C_{t,2t}(\rho_t,a-m+t,b-n+m+t)$.  

Figure \ref{fig:reduction} shows how Remark \ref{rem:reduction} implies Theorem \ref{thm:main1}, since the process of removing a row or column from $\la$, from the $m$ by $n-m$ rectangle, and from $\hkab$ can be repeated until $\la$ has been reduced to $\rho_t$.

\section{Proof of Theorem \ref{thm:main2}}

From this point on, we will write $C_{t,2t}(\rho_t,a,b)$ as $c(t,a,b)$ for simplicity.  Recall the statement of Theorem \ref{thm:main2}:

\begin{theorem3*}
	
	Let $t\geq1$, and let $0\leq a,b\leq t$.  Then:
	
	\[
	c(t,a,b)=(-1)^{a+b+1}\sum\limits_{i=1}^{\min(a,b)}\binom{t-1-i}{a-i}\binom{t-1-i}{b-i}\\
	\]
	
\end{theorem3*}

We need to turn the double alternating sum from Corollary \ref{cor:taut} into a manifestly positive formula.  This requires treating binomial coefficients as polynomials in their own right.  The definition $\binom{x}{k}=\frac{x(x-1)...(x-r+1)}{k!}$ is a degree $k$ polynomial with rational coefficients for any $k\in\N$.  In particular, $x$ can be negative, and the identity $\binom{-x}{k}=(-1)^k\binom{x+k-1}{k}$ relates negative to positive values of this polynomial.  Any binomial coefficient identity proven in a way that assumes $x\in\N$ is still true for other $x$ values, because polynomials that are equal at infinitely many inputs are equal everywhere.  The following lemma can then be proven via combinatorial argument:

\begin{lemma}\label{lemma:kittens}
	For $n,m\in\Z$ and $k\geq0$
	
	\[\binom{n}{k}=\sum\limits_{j=0}^k(-1)^j\binom{n+m-j}{k-j}\binom{m}{j}\]
\end{lemma}

\begin{proof}
	Suppose there are $n$ short-tailed kittens and $m$ long-tailed kittens, and you wish to choose $k$ short-tailed kittens.  Certainly there are $\binom{n}{k}$ ways to do so.  Another way to count it is to take the $\binom{n+m}{k}$ ways to choose $k$ kittens irrespective of tail length, and subtract the number of ways to choose $k$ kittens where at least one has a long tail.  To compute this, we will use the principle of inclusion-exclusion.
	
	Let $K_i$ be the set of all choices of $k$ kittens which include the $i$th long-tailed kitten.  Then $\bigcup\limits_{i=1}^mK_i$ is the set of all choices of $k$ kittens which include any long-tailed kitten, so $\binom{n}{k}=\binom{n+m}{k}-\left|\bigcup\limits_{i=1}^mK_i\right|$.  By the principle of inclusion-exclusion:
	
	\[\left|\bigcup\limits_{i=1}^mK_i\right|=\sum\limits_{j=1}^m(-1)^{j+1}\left(\sum\limits_{1\leq i_1<...<i_j\leq m}\left|K_{i_1}\cap...\cap K_{i_j}\right|\right)\]
	
	Each intersection $K_{i_1}\cap...\cap K_{i_j}$ describes the number of ways to choose $k$ kittens with $j$ of them fixed, so the size of each is $\binom{n+m-j}{k-j}$.  There are $\binom{m}{j}$ choices of $i_1,...,i_j$, so the inner sum evaluates to $\binom{n+m-j}{k-j}\binom{m}{j}$, and thus:
	
	\[\binom{n}{k}=\binom{n+m}{k}-\sum\limits_{j=1}^m(-1)^{j+1}\binom{n+m-j}{k-j}\binom{m}{j}\]
	
	Since $\binom{n+m-j}{k-j}=0$ for $j>k$, the top bound of the summation above can be $k$ instead of $m$ without changing anything, so the right hand side of the equation is equal to $\sum\limits_{j=0}^k(-1)^j\binom{n+m-j}{k-j}\binom{m}{j}$ as desired.
\end{proof}

We can now turn the double alternating sum from Theorem \ref{thm:main2} into a single alternating sum as follows.

\begin{lemma}\label{lemma:outersum}
	For $t\geq1$ and $a,b\geq0$, $$c(t,a,b)=\sum\limits_{i=1}^a(-1)^{b+i+1}\binom{t-1}{i-1}\binom{t-2-a+i}{b-1}$$
\end{lemma}
\begin{proof}
	If $b=0$, both sides of the above are equal to 0, so we can assume $b\geq1$.  Starting with the expression from Corollary \ref{cor:taut}, we have:
	
	\begin{align*}
		c(t,a,b) &= \sum\limits_{i=1}^a\sum\limits_{j=1}^b(-1)^{i+j+1}\binom{a-i+b-j}{a-i}\binom{t-1}{i-1}\binom{t-1}{j-1}\\
		&= \sum\limits_{i=1}^a(-1)^i\binom{t-1}{i-1}\sum\limits_{j=1}^b(-1)^{j+1}\binom{a-i+b-j}{b-j}\binom{t-1}{j-1}\\
		&= \sum\limits_{i=1}^a(-1)^i\binom{t-1}{i-1}\sum\limits_{j=0}^{b-1}(-1)^j\binom{a-i+b-1-j}{b-1-j}\binom{t-1}{j}\\
	\end{align*}

	It therefore suffices to show:
	
	\begin{equation}\label{eqn:singlesum}
		(-1)^{b+1}\binom{t-2-a+i}{b-1}=\sum\limits_{j=0}^{b-1}(-1)^j\binom{a-i+b-1-j}{b-1-j}\binom{t-1}{j}
	\end{equation}
	
	We apply the result of Lemma \ref{lemma:kittens} with with $n=a-i+b-t$ (which may be negative), $m=t-1$, and $k=b-1$ (which is nonnegative since $b\geq1)$, and then use the identity that relates negative to positive values of binomial coefficients:
	
	\[\binom{a-i+b-t}{b-1}=\sum\limits_{j=0}^{b-1}(-1)^j\binom{a-i+b-1-j}{b-1-j}\binom{t-1}{j}\]

	By the identity that relates positive to negative values of binomial coefficients, the above is equivalent to Equation \ref{eqn:singlesum}, which proves the lemma.
\end{proof}

The expression from Lemma \ref{lemma:outersum} is now a single alternating sum, so to end with a manifestly positive formula it remains to remove the alternation.  To do this, we will introduce two new functions $f$ and $g$ and a dummy variable $r$.  The dummy variable will be necessary for an inductive proof to work correctly, but only the $r=0$ case is directly relevant to proving Theorem \ref{thm:main2}.

\begin{defn}\label{def:newparam}
	For $t,a,b,r\in\Z$ with $a\geq2$ and $b\geq1$, define:
	
	\begin{align*}
		f(t,a,b,r) &:= \sum\limits_{i=1}^{a-1}(-1)^{a+i}\binom{t-1}{i-1}\binom{t-2-a+r+i}{b-1}\\
		g(t,a,b,r) &:= -\binom{t-2}{a-2}\binom{t-2+r}{b-1}+\sum\limits_{i=1}^{a-1}(-1)^{a+i+1}\binom{t-2}{i-1}\binom{t-2-a+r+i}{b-2}
	\end{align*}

\end{defn}

\begin{remark}\label{rem:cf}
	By Lemma \ref{lemma:outersum}, \[c(t,a,b)=(-1)^{a+b+1}\binom{t-1}{a-1}\binom{t-2}{b-1}+(-1)^{a+b+1}f(t,a,b,0)\]
\end{remark}

\begin{lemma}\label{lemma:newparam}
	$f(t,a,b,r)=g(t,a,b,r)$ on the entire domain.
\end{lemma}
\begin{proof}
	We show this by induction on $a$.  In the $a=2$ base case, we have:
	
	\begin{align*}
		f(t,2,b,r) &= -\binom{t-1}{0}\binom{t-3+r}{b-1}\\
		&=-\binom{t-3+r}{b-1} \\
		g(t,2,b,r) &= -\binom{t-2}{0}\binom{t-2+r}{b-1}+\binom{t-2}{0}\binom{t-3+r}{b-2} \\
		&= -\binom{t-2+r}{b-1}+\binom{t-3+r}{b-2}
	\end{align*}
	
	These are equal because $-\binom{n-1}{k}=-\binom{n}{k}+\binom{n-1}{k-1}$ for any $n,k$.  Now inductively assume that $f(t,a-1,b,r)=g(t,a-1,b,r)$ for every $t,b,r$ and some particular $a\geq3$.  We extract the $i=a-1$ term of the sum from $f$ to get:
	
	\begin{align*}
			f(t,a,b,r) & =
			\sum\limits_{i=1}^{a-1}(-1)^{a+i}\binom{t-1}{i-1}\binom{t-2-a+r+i}{b-1}\\
			& = -\binom{t-1}{a-2}\binom{t-3+r}{b-1}-\sum\limits_{i=1}^{a-2}(-1)^{a-1+i}\binom{t-1}{i-1}\binom{t-2-a+r+i}{b-1}\\
			& = -\binom{t-1}{a-2}\binom{t-3+r}{b-1}-f(t,a-1,b,r-1)\\
	\end{align*}

	Now by the inductive hypothesis, $f(t,a-1,b,r-1)=g(t,a-1,b,r-1)$.  We use the same binomial identity as in the base case to combine the two terms outside the large sum and then split them up differently to reach $g$.
	
	\begin{align*}
		f(t,a,b,r) & = -\binom{t-1}{a-2}\binom{t-3+r}{b-1}-g(t,a-1,b,r-1)\\
		& = -\binom{t-1}{a-2}\binom{t-3+r}{b-1}+\binom{t-2}{a-3}\binom{t-3+r}{b-1}\\
		& \phantom{==}-  \sum\limits_{i=1}^{a-2}(-1)^{a+i}\binom{t-2}{i-1}\binom{t-2-a+r+i}{b-2}\\
		& = -\binom{t-2}{a-2}\binom{t-3+r}{b-1}+\sum\limits_{i=1}^{a-2}(-1)^{a+i+1}\binom{t-2}{i-1}\binom{t-2-a+r+i}{b-2}\\
		& = -\binom{t-2}{a-2}\binom{t-2+r}{b-1}+\binom{t-2}{a-2}\binom{t-3+r}{b-2}\\
		& \phantom{==} +\sum\limits_{i=1}^{a-2}(-1)^{a+i+1}\binom{t-2}{i-1}\binom{t-2-a+r+i}{b-2}\\
		& = -\binom{t-2}{a-2}\binom{t-2+r}{b-1}+\sum\limits_{i=1}^{a-1}(-1)^{a+i+1}\binom{t-2}{i-1}\binom{t-2-a+r+i}{b-2}\\
		& = g(t,a,b,r)\\
	\end{align*}

	This completes the induction and proves the lemma.
\end{proof}

We are now ready to finish the proof of Theorem \ref{thm:main2}.

\begin{proof}(Theorem \ref{thm:main2})
	The formula for $c(t,a,b)$ given in Corollary \ref{cor:taut} is symmetric in $a$ and $b$, so without loss of generality we may assume $\min(a,b)=a$.  We will induct on $a$.  If $a=0$, both sides of the equation in Theorem \ref{thm:main2} are equal to 0.  If $a=1$, Lemma \ref{lemma:outersum} tells us:
	
	\begin{align*}
		c(t,1,b)=(-1)^b\binom{t-1}{0}\binom{t-2}{b-1}=(-1)^b\binom{t-2}{b-1}
	\end{align*}
	
	Which is also the result of plugging $a=1$ into the formula from Theorem \ref{thm:main2}.   Inductively assume that Theorem \ref{thm:main2} holds for some particular $a-1$ and any $t$ and $b$, so that:
	
	\begin{align*}
		c(t-1,a-1,b-1) &= (-1)^{a+b+1}\sum\limits_{i=1}^{a-1}\binom{t-2-i}{a-1-i}\binom{t-2-i}{b-1-i}\\
		&= (-1)^{a+b+1}\sum\limits_{i=2}^a\binom{t-1-i}{a-i}\binom{t-1-i}{b-i}\\
	\end{align*}

	Starting with the statement of Remark \ref{rem:cf} and using Lemma \ref{lemma:newparam}:
	
	\begin{align*}
		c(t,a,b) & = (-1)^{a+b+1}\binom{t-1}{a-1}\binom{t-2}{b-1}+(-1)^{a+b+1}f(t,a,b,0)\\	
		& = (-1)^{a+b+1}\binom{t-1}{a-1}\binom{t-2}{b-1}+(-1)^{a+b+1}g(t,a,b,0)\\	
		& = (-1)^{a+b+1}\binom{t-1}{a-1}\binom{t-2}{b-1}\\		
		&  \phantom{==}+(-1)^{a+b+1}\left(-\binom{t-2}{a-2}\binom{t-2}{b-1}+\sum\limits_{i=1}^{a-1}(-1)^{a+i+1}\binom{t-2}{i-1}\binom{t-2-a+i}{b-2}\right)\\	
	\end{align*}

	From here, we combine the two loose terms and use the inductive hypothesis, since the summation above is exactly $c(t-1,a-1,b-1)$ by Lemma \ref{lemma:outersum}.
	
	\begin{align*}
		c(t,a,b) & = (-1)^{a+b+1}\binom{t-2}{a-1}\binom{t-2}{b-1}+c(t-1,a-1,b-1)\\
		& = (-1)^{a+b+1}\binom{t-2}{a-1}\binom{t-2}{b-1}+(-1)^{a+b+1}\sum\limits_{i=2}^a\binom{t-1-i}{a-i}\binom{t-1-i}{b-i}\\
		& = (-1)^{a+b+1}\sum\limits_{i=1}^a\binom{t-1-i}{a-i}\binom{t-1-i}{b-i}\\	
	\end{align*}

	This proves the theorem.
\end{proof}

\section{A Relationship Between Classical and Quantum Structure Constants}

Theorem \ref{thm:main0} states that $N_{\la,\hkab}^{\la[1]}$ is the only structure constant in the product $\Oh^\la\Oh^\hkab$ that cannot be computed using the classical K-theoretic Littlewood-Richardson rule.  This rule is purely combinatorial and therefore manifestly positive.  This section of the paper takes inspiration from the classical Littlewood-Richardson rule to give a combinatorial interpretation of the coefficient $N_{\rho_t,\hkab}^{\rho_t[1]}=c(t,a,b)$ as well.

To state the classical K-theoretic Littlewood-Richardson rule, we must first set up some definitions.  If $\la$ and $\mu$ are Young diagrams, define the skew shape $\la\ast\mu$ to be the Young diagram given by placing the bottom left corner of $\mu$ on the top right corner of $\la$, as shown in Figure \ref{fig:laastmu}.

\begin{figure}[]
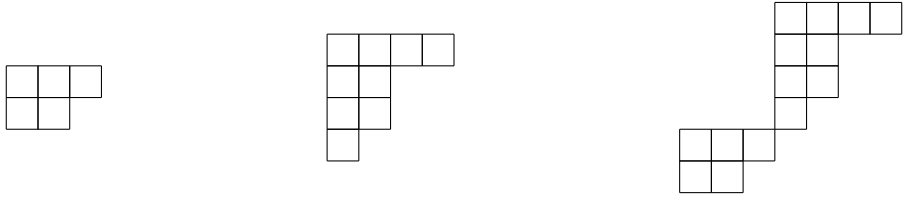

	\[ \tableau{12}
	{	[tlbr]& [tlbr]& [tlbr]\\
		[tlbr]& [tlbr]& []
	}
	\hspace{3cm}	
	\tableau{12}
	{	[tlbr]& [tlbr]& [tlbr]& [tlbr]\\
		[tlbr]& [tlbr]& []& []\\
		[tlbr]& [tlbr]& []& []\\
		[tlbr]& []& []& []
	}
	\hspace{3cm}	
	\tableau{12}
	{	[]& []& []& [tlbr]& [tlbr]& [tlbr]& [tlbr]\\
		[]& []& []& [tlbr]& [tlbr]& []& []\\
		[]& []& []& [tlbr]& [tlbr]& []& []\\
		[]& []& []& [tlbr]& []& []& []\\
		[tlbr]& [tlbr]& [tlbr]& []& []& []& []\\
		[tlbr]& [tlbr]& []& []& []& []& []
	}\]
	
	\caption{From left to right: $\la=(3,2)$, $\mu=(4,2,2,1)$, $\la\ast\mu=(7,5,5,4,3,2)/(3,3,3,3)$}
	\label{fig:laastmu}	
\end{figure}

Given finite nonempty sets $A,B\subset\N$, define $A\leq B$ if $\max(A)\leq\min(B)$, and $A<B$ if $\max(A)<\min(B)$.  A \textit{set-valued tableaux} is a labeling of the boxes in a Young diagram (or a skew diagram) with finite nonempty subsets of $\N$ such that the rows weakly increase from left to right and the columns strictly increase from top to bottom.  Note that a semistandard Young tableaux is a special case of a set-valued tableaux, where each of the sets used has exactly one element.  If $T$ is a set-valuex tableaux, we define $w(T)$ to be the sequence of the numbers in the boxes of $T$, read from bottom to top and then from left to right, where each box is read in increasing order.  Figure \ref{fig:setvaltab} provides an example.

\begin{figure}[]
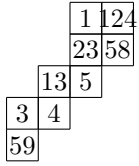

	\[ \tableau{12}
	{	[]& []& [tlbr]1& [tlbr]124\\
		[]& []& [tlbr]23& [tlbr]58\\
		[]& [tlbr]13& [tlbr]5& []\\
		[tlbr]3& [tlbr]4& []& []\\
		[lbr]59& []& []& []
	}\]
	\caption{A set-valued skew tableaux with word $(5,9,3,4,1,3,5,2,3,5,8,1,1,2,4)$}
	\label{fig:setvaltab}	
\end{figure}

A word $w=(w_1,w_2,...,w_k)$ satisfies the \textit{reverse lattice condition} if, for each $i$, either $w_i=1$ or $w_i$ is followed by more instances of $w_i-1$ than of $w_i$.  The \textit{content} of $w$ is the sequence $(c_1,...,c_r)$, where $c_i$ is the number of $i$'s in $w$.  Note that if $w$ satisfies the reverse lattice condition, the content of $w$ is a partition.

The classical K-theoretic Littlewood-Richardson rule is:

\begin{thm}[\cite{buch_gamma}]\label{thm:ktlr}
	Given classical shapes $\la,\mu,\nu$, the structure constant $N_{\la,\mu}^\nu$ of $\K(X)$ is equal to $(-1)^{|\nu|-|\la|-|\mu|}$ times the number of set-valued tableaux $T$ of shape $\la\ast\mu$ such that $w(T)$ is a reverse lattice word with content $\nu$.
\end{thm}

As a comparison to $c(t,a,b)$, let $\nu$ be the quantum shape in $\Z^2/(t,-t)$ given by $(t,t,t-1,t-2,...,3,2)$.  In other words, $\nu$ is $\rho_t[1]$ with one box removed.  By Theorem \ref{thm:main0}, $N_{\rho_t,\hkab}^\nu$ can be computed using the ordinary K-theoretic Littlewood-Richardson rule.

\begin{prop}
	In $\QK(\Gr(t,2t))$, for $\nu=(t,t,t-1,t-2,...,3,2)$, $N_{\rho_t,\hkab}^\nu=(-1)^{a+b}\binom{t-2}{a-1}\binom{t-2}{b-1}$
\end{prop}
\begin{proof}
	We need to count the number of ways to create a set-valued tableaux with shape $\hkab\ast\rho_t$ and content $\nu$, subject to the reverse lattice condition.  The reverse lattice condition forces each row of $\rho_t$ to be filled with exactly its row number, since
	the furthest rightmost box in the top row must have a 1, which forces the boxes to
	the left of it to also have 1s, and the rest of the rows are filled similarly. This is
	demonstrated in Figure \ref{fig:fillrhot}.
	
	\begin{figure}[]
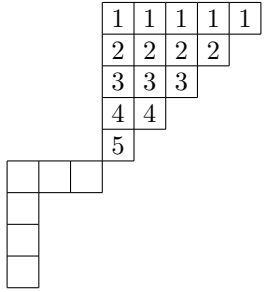

		\[ \tableau{12}
		{	[]& []& []& [tlbr]1& [tbr]1& [tbr]1& [tbr]1& [tbr]1\\
			[]& []& []& [lbr]2& [br]2& [br]2& [br]2& []\\
			[]& []& []& [lbr]3& [br]3& [br]3& []& []\\
			[]& []& []& [lbr]4& [br]4& []& []& []\\
			[]& []& []& [lbr]5& []& []& []& []\\
			[tlbr]& [trb]& [tbr]& []& []& []& []& []\\
			[lbr]& []& []& []& []& []& []& []\\
			[lbr]& []& []& []& []& []& []& []\\
			[lbr]& []& []& []& []& []& []& []
		}\]
		\caption{Any set-valued tableaux $T$ of shape $(3\backslash 2)\ast\rho_5$ with $w(T)$ a reverse lattice word must have $\rho_5$ filled with singletons as above.}
		\label{fig:fillrhot}	
	\end{figure}

	Now we need to fill the $\hkab$ section of the tableaux with content $(1,2,2,...,2)$.  The reverse lattice condition implies that no number can be contained in more than one box in the top row, and the column must strictly increase, so each of the numbers $2$ through $t$ must be in one box in the row of the hook and one box in the column of the hook.  This forces the corner to be assigned precisely $\{1\}$.  Then by stars and bars there are $\binom{t-2}{a-1}$ ways to fill the column and $\binom{t-2}{b-1}$ ways to fill the row.  The column and the row are independent of each other, so there are $\binom{t-2}{a-1}\binom{t-2}{b-1}$ set valued tableaux with shape $\hkab\ast\rho_t$ and content $\nu$.
\end{proof}

Notably, $\binom{t-2}{a-1}\binom{t-2}{b-1}$ is the highest degree term of the sum for $c(t,a,b)$.  Other terms are smaller, and so can be considered as a subset of the tableaux counted above.  The precise form of the terms suggests a combinatorial description of $c(t,a,b)$.  There are many variants on this description that work equally well, but the version in Proposition \ref{prop:combodesc} below is the easiest to state.

\begin{prop}\label{prop:combodesc}
	$c(t,a,b)$ is equal to $(-1)^{a+b+1}$ times the number of pairs $(i,T)$, where $i\in\N$, $T$ is a set-valued tableaux of content $(t,t,t-1,t-2,...,3,2)$ and shape $\hkab\ast\rho_t$, $w(T)$ is a reverse lattice word, and for each $j\leq i$, every box which contains $j$ contains precisely the singleton set $\{j\}$.
\end{prop}
\begin{proof}
	As previously established, every set in the $\rho_t$ section of the tableaux is a singleton, as is the $\{1\}$ in the corner of the $\hkab$ section.  If every instance of each $j\leq i$ needs to be a singleton, this entirely determines the sets in the first $i$ boxes in the column and the row of the hook.  Again by stars and bars, there are $\binom{t-1-i}{a-i}$ ways to fill the rest of the column and $\binom{t-1-i}{b-i}$ ways to fill the rest of the row.  Thus the total number of pairs $(i,T)$ is exactly $\sum\limits_{i=1}^{\min(a,b)}\binom{t-1-i}{a-i}\binom{t-1-i}{b-i}=c(t,a,b)$, as desired.
\end{proof}

\bibliographystyle{amsplain}
\bibliography{bib.bib}
\end{document}

%% file: preamble.tex
\def\newthm#1#2{\newtheorem{#1}[dummy]{#2}%
  \expandafter\def\csname#2\endcsname##1{\hyperref[#1:##1]{#2~\ref*{#1:##1}}}}
\newthm{thm}{Theorem}
\newthm{lemma}{Lemma}
\newthm{prop}{Proposition}
\newthm{cor}{Corollary}
\newthm{conj}{Conjecture}
\newthm{cons}{Consequence}
\newthm{fact}{Fact}
\newthm{obs}{Observation}
\newthm{guess}{Guess}
\theoremstyle{definition}
\newthm{defn}{Definition}
\newthm{remark}{Remark}
\newthm{example}{Example}
\newthm{question}{Question}
\newthm{notation}{Notation}

\newtheorem*{theorem2*}{Theorem 2}
\newtheorem*{theorem3*}{Theorem 3}

\makeatletter
\def\namedlabel#1#2{\begingroup
    #2%
    \def\@currentlabel{#2}%
    \phantomsection\label{#1}\endgroup
}
\makeatother

\newcommand{\Section}[1]{\hyperref[sec:#1]{Section~\ref*{sec:#1}}}
\newcommand{\Table}[1]{\hyperref[tab:#1]{Table~\ref*{tab:#1}}}
\newcommand{\eqn}[1]{\hyperref[eqn:#1]{(\ref*{eqn:#1})}}
\newcommand{\Figure}[1]{\hyperref[fig:#1]{Figure~\ref*{fig:#1}}}

\DeclareMathOperator{\Gr}{Gr}

\DeclareMathOperator{\QK}{QK}
\DeclareMathOperator{\K}{K}

\DeclareMathOperator{\Row}{rows}
\DeclareMathOperator{\Col}{cols}

\DeclareMathOperator{\ch}{\ch}
\newcommand{\Oh}{\mathcal{O}}

\usepackage{bm}

\newcommand{\C}{{\mathbb C}}

\newcommand{\Z}{{\mathbb Z}}
\newcommand{\N}{{\mathbb N}}

\newcommand{\al}{{\alpha}}
\newcommand{\be}{{\beta}}

\newcommand{\la}{{\lambda}}

\newcommand{\hk}[2]{\mathcal{O}^{(#1 \backslash #2)}}
\newcommand{\hkab}{{(a \backslash b)}}

\newcommand{\ignore}[1]{}